\documentclass[leqno,12pt]{article}
\usepackage{inputenc}
\usepackage[english]{babel}
\usepackage{amsmath}
\usepackage{amssymb}
\usepackage{amsthm}
\date{}
\newtheorem{thm}{Theorem}[section]
\newtheorem{lem}[thm]{Lemma}
\newtheorem{prop}[thm]{Proposition}
\newtheorem{cor}[thm]{Corollary}

\theoremstyle{definition}
\newtheorem{ex}[thm]{\it Example}
\newtheorem{exs}[thm]{\it Examples}
\newtheorem{rem}[thm]{\it Remark}
\newtheorem{rems}[thm]{\it Remarks}
\newtheorem{disc}[thm]{\it Discussion}

\newtheorem{defin}[thm]{Definition}
\numberwithin{equation}{section}

\title{When the associated graded ring of a semigroup ring is Complete Intersection}

\begin{document}
\newcommand{\Hi}{\mathrm {Hilb}}
\newcommand{\card}{\mathrm {card}}
\newcommand{\ap}{\mathrm {Ap}}
\newcommand{\ord}{\mathrm {ord}}
\newcommand{\mapM}{\mathrm {maxAp_M}}
\newcommand{\map}{\mathrm {maxAp}}
\newcommand{\gr}{\mathrm{gr}}
\author{M. D'Anna\thanks{{\em email} mdanna@dmi.unict.it} \and V. Micale \and A. Sammartano}

\maketitle

\begin{abstract}
\noindent Let $(R,\mathfrak{m})$ be the  semigroup ring associated
to a numerical semigroup $S$. In this paper we study the property
of its associated graded ring $\gr_{\mathfrak{m}}(R)$ to be
Complete Intersection. In particular, we introduce and
characterize $\beta$-rectangular and $\gamma$-rectangular Ap\'ery
sets, which will be the fundamental concepts of the paper and will
provide, respectively, a sufficient condition and a
characterization for $\gr_{\mathfrak{m}}(R)$ to be Complete
Intersection. Then we use these notions to give four equivalent
conditions for $\gr_{\mathfrak{m}}(R)$ in order to be Complete
Intersection.
\medskip

\end{abstract}

\section*{Introduction}

Let $(R,\mathfrak{m})$ be a Noetherian local ring with
$|R/\mathfrak{m}|=\infty$ and let
$\gr_{\mathfrak{m}}(R)=\oplus_{i\geq 0}\mathfrak{m}^i/
\mathfrak{m}^{i+1}$ be the associated graded ring of $R$ with
respect to $\mathfrak{m}$. The study of the properties of
$\gr_{\mathfrak{m}}(R)$ is a classical subject in local algebra,
not only in the general $d$-dimensional case, but also under
particular hypotheses (that allow to obtain more precise results).
One main problem in this context is to estimate the depth of
$\gr_{\mathfrak{m}}(R)$ and to understand when this ring is a
Cohen-Macaulay ring (see, e.g., \cite {RV}, \cite {Sally} and
\cite {VV}). In connection to this problem, it is natural to
investigate if $\gr_{\mathfrak{m}}(R)$ is a Buchsbaum ring (see
\cite{Go1}, \cite{Go2}), a Gorenstein ring or if it is Complete
Intersection (see \cite{HKU}).

In this paper we are interested in the properties of
$\gr_{\mathfrak{m}}(R)$, when $R$ is a numerical semigroup ring.
The study of numerical semigroup rings is motivated by their
connection to singularities of monomial curves and by the
possibility of translating algebraic properties into numerical
properties. However, even in this particular case, many
pathologies occur, hence these rings are also a great source of
interesting examples.

In the numerical semigroup case, the Cohen-Macaulayness of
$\gr_{\mathfrak{m}}(R)$ has been extensively studied (see, e.g.,
\cite {Ga}, \cite {MPT} and \cite {SH}); recently, different
authors studied the Buchsbaumness and the Gorensteinness of
$\gr_{\mathfrak{m}}(R)$ (see \cite{Br}, \cite{DMM}, \cite{DMS} and
\cite{Sa}).

In this paper we investigate when $\gr_{\mathfrak{m}}(R)$ is
Complete Intersection. About this problem not much is known (see
\cite {BDF}). When the embedding dimension of $R$ is small, it is
possible to list the generators of the defining ideals of $R$ and
of $\gr_{\mathfrak{m}}(R)$ (as it is done in \cite{H} and
\cite{SH} when the embedding dimension is $2$ or $3$), but, as
soon as the number of generators of $\mathfrak{m}$ increases, the
computations become too huge. On the other hand, a useful tool to
study the general case (when $\mathfrak{m}$ is $n$-generated) is
the so called Ap\'ery set of the semigroup. The properties of the
Ap\'ery set reveal much information on the Cohen-Macaulayness, the
Buchsbaumness and the Gorensteinness of $\gr_{\mathfrak{m}}(R)$.
In this paper, using the Ap\'ery set, we are able to characterize
when $\gr_{\mathfrak{m}}(R)$ is Complete Intersection.\\

The structure of the paper is the following. In Section 1 we fix
the notation and give some preliminaries about numerical
semigroups, semigroup rings associated to a numerical semigroup
and their associated graded ring. In particular, we qualitatively
discuss the form of the elements in the defining ideals of $R$, of
its quotient $\overline R$ modulo an element of minimal value and
of their associated graded rings (cf. Discussion \ref{monomials}).

In Section 2 we define two sets of integers $\beta_i$ and
$\gamma_i$, that yield to the definition and the characterization
of two classes of numerical semigroups: semigroups with
$\beta$-rectangular and $\gamma$-rectangular Ap\'ery set (cf.
Definition \ref{1}, Theorem \ref{d} and Theorem \ref{charact}).
These notions will provide, respectively, a sufficient condition
and a characterization for $\gr_{\mathfrak{m}}(R)$ to be Complete
Intersection. These classes are strictly connected and it is
useful to study both of them together. The idea in these
definitions is to use the ``shape" of the Ap\'ery set of the
semigroup in order to find integers that give information on the
degree and on the nature of the generators of the defining ideals
of $R$, of its quotient $\overline R$ modulo an element of minimal
value and of their associated graded rings.

In Section 3 we prove the main theorem of the paper (cf. Theorem
\ref{7}), that is the characterization of those numerical
semigroup rings whose associated graded ring is Complete
Intersection. More precisely, we give four equivalent conditions
for $\gr_{\mathfrak{m}}(R)$ in order to be Complete Intersection.
To obtain this result we need to deepen carefully the nature of
the elements of the defining ideals of $R$, of its quotient
$\overline R$ modulo an element of minimal value and of their
associated graded rings, using the integers $\beta_i$ and
$\gamma_i$ introduced in the previous section (cf. Discussion
\ref{monomials2} and Lemma \ref{JCI}). Finally, we give an
alternative proof of a sufficient condition for
$\gr_{\mathfrak{m}}(R)$ being Complete Intersection presented in
\cite {BF} (cf. Corollary \ref{7a}) and we briefly study the case
of embedding dimension $3$ (cf. Theorem \ref{3gen}).

The computations made for this paper are performed by using the GAP
system \cite{GAP} and, in particular, the NumericalSgps package
\cite{NS}.

\section{Preliminaries }

Let $\mathbb{N}$ denote the set of natural numbers, including $0$.
A \emph{numerical semigroup} is a submonoid $S$ of the monoid
$(\mathbb{N},+)$ with finite complement in it. Each numerical
semigroup $S$ has a natural partial ordering $\preceq$ where, for
every $s$ and $t$ in $S$, $s\preceq t$ if there is an element
$u\in S$ such that $t=s+u$. The set $\{g_i \}$ of the minimal
elements in the poset $(S \setminus \{0\}, \preceq)$ is called
\emph{minimal set of generators} for $S$; indeed all the elements
in $S$ are linear combinations, with coefficients in $\mathbb{N}$,
of minimal elements. Note that the set $\{g_i \}$  is finite since
for any $s \in S , \, s \ne 0$, we have  $g_i \not\equiv g_j\,
\pmod{s}$ if $i \ne j$. A numerical semigroup minimally generated
by $g_1 < g_2 < \ldots < g_\nu$ is denoted by $ \langle g_1, g_2,
\ldots, g_\nu \rangle$; the condition $| \mathbb{N} \setminus S| <
\infty$ is equivalent to $\gcd(g_1, \ldots, g_\nu)=1$.

There are several invariants associated to a numerical semigroup
$S$. The integer $m=g_1= \min \{ s \in S, \, s > 0\}$ is called
\emph{multiplicity}, while the minimal number of generators $\nu$
is called \emph{embedding dimension}; it is well known that $ \nu
\leq m$. Finally, the integer $f = \max \{ z \in \mathbb{Z},\, z
\notin S \}$ is called \emph{Frobenius number} of $S$.\\

Following the notation in \cite{BF}, we denote by
$\ap(S)=\{\omega_0,\dots,\omega_{m-1}\}$ the {\em Ap\'ery set} of
$S$ with respect to $m$, that is, the set of the smallest elements
in $S$ in each congruence class modulo $m$. More precisely,
$\omega_0=0$ and $\omega_i=\min\{s\in S\mid s\equiv i \pmod{m}\}$.
The largest element in the Ap\'ery set is always $f+m$.

Very often we will use the following result.

\begin{prop}[\cite{FGH}, Lemma 6]\label{1.1}
Let $t\in\ap(S)$ and $u \preceq t$, then $u\in\ap(S)$.
\end{prop}

A numerical semigroup $S$ is called {\em symmetric} if $f-x \notin
S$ implies that $x \in S$  for every integer $x$ (notice that the
converse is true for every numerical semigroup $S$).

\begin{prop}[\cite{RG}, Corollary 4.12]\label{3}
$S$ is symmetric if and only if $f+m$ is the unique maximal
element in $(\ap(S), \preceq)$.
\end{prop}

An {\em ideal} of a semigroup $S$ is a nonempty subset $H$ of $S$
such that $H+S\subseteq H$. The ideal $M=\{s\in S \mid s\neq 0\}$
is called  {\em maximal ideal of $S$}. It is straightforward to
see that, if $H$ and $L$ are ideals of $S$, then $H+L=\{h+l | \ h
\in H, \ l \in L\}$ and $kH(= H+\dots+H$, $k$ summands, for $k\geq
1$) are also ideals of $S$.
\\

Let $\Bbbk$ be an infinite field; the rings
$R=\Bbbk[[t^S]]=\Bbbk[[t^{g_{1}},\dots,t^{g_{\nu}}]]$ and
$R=\Bbbk[t^S]_{\mathfrak{m}}$ are called the {\em numerical semigroup
rings} associated to $S$. The ring $R$ is a one-dimensional local
domain, with maximal ideal
$\mathfrak{m}=(t^{g_{1}},\dots,t^{g_{\nu}})$ and quotient field
$\Bbbk((t))$ and $\Bbbk(t)$, respectively. In both cases the associated
graded ring of $R$ with respect to $\mathfrak{m}$,
$\gr_{\mathfrak{m}}(R)=\oplus_{i\geq 0}\mathfrak{m}^i/
\mathfrak{m}^{i+1}$, is the same. From now on, we will assume that
$R=\Bbbk[[t^S]]$, but the other case is perfectly analogous.

Let $(A,\mathfrak{n})$ be the local ring of formal power series
$\Bbbk[[x_1,\dots,x_{\nu}]]$ and let $\varphi:A\longrightarrow R$ be the
map defined by $\varphi(x_i)=t^{g_{i}}$. Clearly $R=A/I$ and
$\mathfrak{m}=\mathfrak{n}/I$, with $I=\ker\varphi$. Notice that
$I$ is a binomial ideal generated by all the elements of the form
\begin{equation}\tag{$*$}
x_1^{j_1}\cdot \dots \cdot x_{\nu}^{j_{\nu}}-x_1^{h_1}\cdot \dots
\cdot x_{\nu}^{h_{\nu}},
\end{equation}
with $j_1g_1+\dots+j_{\nu}g_{\nu}=
h_1g_1+\dots+h_{\nu}g_{\nu}$.

It is well known that this presentation induces a presentation of
the corresponding associated graded rings:
$$\psi:\gr_{\mathfrak{n}}(A)\longrightarrow
\gr_{\mathfrak{m}}(R),$$ where the kernel is the initial ideal of $I$,
i.e. the ideal $I^*$ generated by the initial forms of the
elements of $I$; hence $\gr_{\mathfrak{m}}(R) \cong
\gr_{\mathfrak{n}}(A)/I^* \cong \Bbbk[x_1,\dots,x_{\nu}]/I^*$
canonically.

Notice that $I^*$ is an homogenous ideal generated by all the
monomials of the form $x_1^{j_1}\cdot \dots \cdot
x_{\nu}^{j_{\nu}}$ coming from a binomial $(*)$ for which
$j_1+\dots+j_{\nu}< h_1+\dots+h_{\nu}$ and by all the binomials
$(*)$ such that $j_1+\dots+j_{\nu} = h_1+\dots+h_{\nu}$.
%
%
%

We denote by $\mu(\cdot)$ the minimal number of generators of an
ideal. The ring $R$ is Complete Intersection if $\mu(I)=\nu-1$ and
the associated graded ring $\gr_{\mathfrak{m}}(R)$ is {\em
Complete Intersection} if $\mu(I^*)=\nu-1$. It is well known that,
if $\gr_{\mathfrak{m}}(R)$ is Complete Intersection, then also $R$
is Complete Intersection.

Numerical semigroups for which $R$ is Complete Intersection are
well known (and they are called Complete Intersection numerical
semigroups; for the de\-finition see, e.g., \cite{RG}). We are
interested in studying when $\gr_{\mathfrak{m}}(R)$ is Complete
Intersection.
Let $\overline{R}=R/(t^m)$ and $\overline{G}=
\gr_{\overline{\mathfrak m}} (\overline{R})$, where
${\overline{\mathfrak m}}$ is the maximal ideal of $\overline{R}$.

\begin{rem}\label{4}
In our hypotheses, it is clear that $\gr_\mathfrak{m}(R)$ is
Complete Intersection if and only if $\overline{G}$ is Complete
Intersection and $\gr_\mathfrak{m}(R)$ is Cohen-Macaulay (i.e., as
it is proved in \cite{Ga}, $(t{^m})^*\in \mathfrak m/\mathfrak
m^2$ is not a zero-divisor in $\gr_\mathfrak{m}(R)$).
In fact, in
case $\gr_\mathfrak{m}(R)$ is Cohen-Macaulay, from the isomorphism $R/(t^m) \cong \overline{R}$ we
get the isomorphism $\gr_\mathfrak{m}(R)/(({t^m})^*) \cong
\overline{G}$. We notice also that, in general, there is a
surjective homomorphism of graded rings $\gr_\mathfrak{m}(R)
\rightarrow \overline{G}$, whose kernel is the initial ideal of
$(t^m)$ in $\gr_\mathfrak{m}(R)$.
\end{rem}

\begin{rem}\label{1'}
We note that $\overline{R}= \langle \overline{t^{\omega_i}} \, |
\, \omega_i \in \ap(S)\rangle_\Bbbk$, since
\[ \overline{t^s}=\overline{0}\mbox{ in }\overline{R} \Longleftrightarrow t^s \in (t^m)
\Longleftrightarrow s-m \in S. \] We also have
$\overline{G}=\overline{R}$ as $\Bbbk$-vector spaces (but not as
rings) since a nonzero monomial in $\overline{R}$ is still nonzero
in $\overline{G}$.
\end{rem}

\begin{disc}\label{monomials}
Recalling the isomorphism $R \cong \Bbb\Bbbk[[x_1,x_2,\dots, x_{\nu}]]/I$,
where $\overline {x_i}$ corresponds to $t^{g_{i}}$ (hence
$\overline {x_1}$ corresponds to $t^m$), we have the isomorphism
$\overline R \cong \Bbb\Bbbk[[x_2,x_3,\dots, x_{\nu}]]/H$.
More precisely $H$ is the kernel of the homomorphism defined by
$x_i \mapsto \overline{t^{g_{i}}}$ and it is generated by all the
binomials of the form
\begin{equation}\tag{$**$}
x_2^{j_2}\cdot \dots \cdot x_{\nu}^{j_{\nu}}-x_2^{h_2}\cdot \dots
\cdot x_{\nu}^{h_{\nu}},
\end{equation}
with $j_2g_2+\dots+j_{\nu}g_{\nu}=
h_2g_2+\dots+h_{\nu}g_{\nu} \in \ap(S)$, and by all the monomials
of the form
\begin{equation}\tag{$***$}
x_2^{j_2}\cdot \dots \cdot x_{\nu}^{j_{\nu}},
\end{equation}
where
$j_2g_2+\dots+j_{\nu}g_{\nu} \notin \ap(S)$.

It follows that $\overline{G} =\gr_{\overline{\mathfrak m}}
(\overline{R}) \cong \Bbb\Bbbk[x_2, x_3, \dots, x_\nu] / J$, where $J$ is
the kernel of the homomorphism defined by $x_i \mapsto
\overline{t^{g_{i}}}$ (where now $\overline{t^{g_{i}}}$ is viewed
as an element of $\overline G$) and it is the initial ideal of
$H$. Hence $J$ is a binomial ideal generated by all the the
binomials of the form
\begin{equation}\tag{$+$}
x_2^{j_2}\cdot \dots \cdot x_{\nu}^{j_{\nu}}-x_2^{h_2}\cdot \dots
\cdot x_{\nu}^{h_{\nu}},
\end{equation}
with $j_2g_2+\dots+j_{\nu}g_{\nu}=
h_2g_2+\dots+h_{\nu}g_{\nu} \in \ap(S)$ and $j_1+\dots+j_{\nu} =
h_1+\dots+h_{\nu}$, and by all the monomials of the form
\begin{equation}\tag{$++$}
x_2^{j_2}\cdot \dots \cdot x_{\nu}^{j_{\nu}},
\end{equation}
where either
$j_2g_2+\dots+j_{\nu}g_{\nu} \notin \ap(S)$ or
$j_2g_2+\dots+j_{\nu}g_{\nu} \in \ap(S)$ and there exist $h_2,
\dots, h_{\nu} \in \mathbb N$, such that
$j_2g_2+\dots+j_{\nu}g_{\nu}= h_2g_2+\dots+h_{\nu}g_{\nu}$ and
$j_2+\dots+j_{\nu}< h_2+\dots+h_{\nu}$.

In particular, let $j=j_2+\dots+j_{\nu}$; then a binomial of the
form $(+)$ is not necessary as generator of $J$, if
$x_2^{j_2}\cdot \dots \cdot x_{\nu}^{j_{\nu}}\in (x_2, \dots ,
x_{\nu})^{j+1}$.
Furthermore, the monomial $x_2^{j_2}\cdot \dots \cdot
x_{\nu}^{j_{\nu}}\in (x_2, \dots , x_{\nu})^j \setminus (x_2,
\dots , x_{\nu})^{j+1}$, such that $j_2g_2+\dots+j_{\nu}g_{\nu}
\in \ap(S)$, does not belong to $J$.

Finally, since the Krull dimension of $\overline{G}$ is $0$, we
must have $\mu(J)\geq \nu -1$. Hence $\overline{G}$ is Complete
Intersection if and only if $ \mu(J)=\nu-1$.
\end{disc}

From the previous remarks and discussion, it is clear why, to
study the Complete Intersection property for $\gr_\mathfrak{m}(R)$
it is necessary to study the Ap\'ery set of $S$.
When
$\overline {G}$ is Complete Intersection, its Hilbert function is
completely determined by the degree of the generators of $J$ and
its dimension as $\Bbbk$-vector space is the product of these degrees.
Hence we will have to determine these degrees using numerical
conditions; moreover, since monomials in $\overline G$ correspond
bijectively to elements of the Apery set, we will have to
determine the ``shape" of $\ap(S)$ corresponding to $\overline G$
Complete Intersection.

\section{
$\beta$- and $\gamma$-rectangular Ap\'ery Sets}

Within this section we introduce
two sets of integers and two corresponding classes of numerical
semigroups, defined via the shape of the Ap\'ery set. The first
class will provide a sufficient condition for $\overline G$ to be
complete intersection, while the second one will give a
characterization.
\\

Given a numerical semigroup $S= \langle g_1, g_2, \ldots, g_\nu
\rangle$ and $s,\,t \in S$,
we recall that
$s \preceq t$ if there exists $u \in S$ such that $s+u=t$.
Now we want to define another partial ordering on $S$ as in \cite{Br}.
If $s\in S$ and $M=S\setminus\{0\}$ then there exists a unique $h\in \mathbb{N}$
such that $s\in hM\setminus (h+1)M$;
this integer is defined as
the {\em order} of $s$ and we will write $\ord (s)=h$.
Given $s,t\in S$, we say that $s\preceq_M t$ if there
exists $u \in S$ such that $s+u=t$ (hence $s\preceq t$) and $\ord
(s)+\ord (u)=\ord(t)$.
The partial order $\preceq_M$ is particularly helpful in the study of the associated graded ring.

The sets of maximal elements of $\ap(S)$ with respect to $\preceq$
and $\preceq_M$ are denoted with $\map(S)$ and $\mapM(S)$,
respectively.

\begin{rem}
We note that $\map(S) \subseteq \mapM(S)$ and the inclusion can be strict.
For
example, let $S=\langle 8,9,15\rangle$. The only maximal ele\-ment
in $\ap(S)$ with respect to $\preceq$ is $45$. Anyway
$\mapM(S)=\left\{30,45\right\}$. Note that
$\ord(45)=5>3=\ord(30)+\ord(15)$.
\end{rem}

A numerical semigroup $S$ is called $M$-{\em pure} if every
element in $\mapM(S)$ has the same order.
$M$-pure symmetric semigroups are characterized in a similar way to symmetric semigroups:
\begin{prop}[\cite{Br}, Proposition 3.7]\label{crede}
A semigroup $S$ is $M$-pure symmetric if and only if $\omega \preceq_M f+m$, for every $\omega \in
\ap(S)$.
\end{prop}

Every element $s \in S$ can be written, not necessarily in a
unique way,
 as $s=\lambda_1g_1+\dots +\lambda_{\nu}g_{\nu}$;
we call this combination of the generators a {\em representation}
of $s$. Throughout the paper, we call ``representation" both the
expression  $s=\lambda_1g_1+\dots +\lambda_{\nu}g_{\nu}$ and the
tuple $(\lambda_1,\lambda_2, \ldots, \lambda_\nu)$. We say that an
element $s\in S$ has a {\em unique representation} if it can be
written in a unique way as a linear combination of $g_1, g_2,
\ldots, g_{\nu}$. Notice that, by definition of Ap\'ery set,
 an
element $\omega \in \ap(S)$ can have only representations where
$g_1$ does not appear.

A representation of an element $s \in S$ as
$s=\lambda_1g_1+\lambda _2g_2+\dots+\lambda_\nu g_\nu$ is called
{\em maximal} if $\lambda_1+\lambda _2+\dots+\lambda_\nu=\ord (s)$.
This kind of representations and, in particular, the number of
maximal representations of elements of $S$ have been studied in \cite{BHJ}.
We say that $\ap(S)$ is of
{\em unique maximal expression} if every $\omega \in \ap(S)$ has a
unique maximal representation.
\\

Let us  define now the following integers, for every $i=2 \dots,
\nu$:
\begin{eqnarray*}
\beta_i &=&\max\{ h \in \mathbb{N} \, | \, hg_i \in \ap(S) \text{ and } \ord(h g_i)=h\};\\
\gamma_i &=&\max\{ h \in \mathbb{N} \, | \, hg_i \in \ap(S), \, \ord(h g_i)=h \text{ and }\\
& & \hspace*{2.3 cm} hg_i \text{ has a unique maximal representation}\}.
\end{eqnarray*}
Notice that,
in the second definition,
from $\ord(hg_i)=h$
it follows that $h g_i$
must be the unique maximal representation.
The following proposition is straightforward:

\begin{prop}\label{sco}
For each index $i=2, \ldots, \nu,$ we have $\gamma_i \leq
\beta_i$.
\end{prop}

\begin{exs}\label{liberi}
Let $S=\langle 8,10,15\rangle$; so $\ap(S)= \{ 0, 10, 15, 20, 25,
30, 35, 45\}$. Since $30, 45 \in \ap(S)$ and $40, 60 \notin
\ap(S)$ then $\beta_2, \beta_3 \leq 3$. We have the double
representation $30= 3 \cdot 10 = 2 \cdot 15$, implying that
$\ord(2\cdot 15)=3
>2$; hence $\beta_2=3$ and $\beta_3=1$. It is easy to check that
every element in $\ap(S)$ has a unique maximal representation,
thus $\gamma_2=\beta_2=3$ and $\gamma_3=\beta_3=1$.

Let $S=\langle 7,9,10,11,12\rangle$; we have $\ap(S)= \{ 0, 9, 10,
11, 12, 20, 22 \}$. Analo\-gously to the previous example, we have
$\beta_2=1, \, \beta_3=2, \, \beta_4=2, \beta_5=1$. The only
double representations in $\ap(S)$ are $20=9+11=10+10$ and
$22=10+12=11+11$, and they are all maximal. In particular, it
follows $\gamma_2=\gamma_3=\gamma_4=\gamma_5=1$.
\end{exs}

In correspondence to the two families of numbers introduced above,
we can define the following two sets:
\begin{eqnarray*}
B &=& \Big\{\sum_{i=2}^{\nu} \lambda_i g_i \, | \, 0 \leq \lambda_i \leq \beta_i \Big\};\\
\Gamma &=& \Big\{\sum_{i=2}^{\nu} \lambda_i g_i \, | \, 0 \leq \lambda_i \leq \gamma_i \Big\}.
\end{eqnarray*}
The  sets $B,\,\Gamma,$ consist of the  elements of $S$
representable via a tuple $(\lambda_2,\ldots, \lambda_\nu)$
belonging to the hyper-rectangle of $\mathbb{N}^{\nu-1}$ whose
vertices are respectively given by $\beta_i$ and $\gamma_i$. By
Proposition \ref{sco}, it  follows that $\Gamma \subseteq B$.
Notice also that, as can be easily seen by the previous examples,
since elements in $B$ and $\Gamma$ can have more than one
representation, $|B| \leq \prod_{i=2}^{\nu}(\beta_i+1)$ and
$|\Gamma| \leq \prod_{i=2}^{\nu}(\gamma_i+1)$.
\\

It is natural to ask  how the  sets $B$ and $\Gamma,$ are related
to the Ap\'ery set. The inclusion $\ap(S)\subseteq B$ is always
true, as we can deduce from the next lemma:


\begin{lem}\label{aurora}
Let $\omega \in \ap(S)$ and let $\omega= \sum_{i=2}^\nu \lambda_i g_i$ be a maximal representation.
Then $ \lambda_i \leq \beta_i$ for each $i$.
\end{lem}
\begin{proof}
By Lemma \ref{1.1} we have $\lambda_i g_i \in \ap(S)$. If there is
a index $i$ such that $\beta_i < \lambda_i$, then $\ord(\lambda_i
g_i) > \lambda_i$, by definitions of $\beta_i$. It follows that
$\ord(\sum_{i=2}^\nu \lambda_i g_i) > \sum_{i=2}^\nu \lambda_i$
and thus it is not a maximal representation of $\omega$;
contradiction.
\end{proof}

Notice also that an integer in $B$ could have more different
maximal representations, as can be seen in the same example as
above $S=\langle 7,9,10,11,12\rangle$.

We want to show that actually the stronger inclusion $\ap(S)
\subseteq \Gamma$ holds; this is a consequence of the following
useful result, which is somehow analogous to Lemma \ref{aurora}.
In what follows, we denote with \textsf{lex} and \textsf{grlex}
respectively the usual lexicographic order and graded
lexicographic order in $\mathbb{N}^{\nu-1}$.

\begin{lem}\label{dolls}
Let $\omega \in \ap(S)$ and set
\[ \mathcal{R}= \Big\{ (\lambda_2, \ldots, \lambda_{\nu}) \in \mathbb{N}^{\nu-1}\, \Big|\,
 \sum_{i=2}^{\nu} \lambda_i g_i=\omega \text{ and } \sum_{i=2}^{\nu} \lambda_i = \ord(\omega) \Big\}\]
i.e. the set of maximal representations of $w$.
Let $(\mu_2, \ldots, \mu_\nu)$ be the maximum in  $\mathcal{R}$ with respect to \textsf{lex},
then we have $\mu_i \leq \gamma_i$ for each $i$.
\end{lem}
\begin{proof}
By Lemma \ref{aurora} we have $\mu_i \leq \beta_i$. Let us suppose
that there exists an index $i$ such that $ \gamma_i < \mu_i \leq
\beta_i$, and take the minimum $i$ with this property. By
definitions of $\beta_i, \, \gamma_i$, the fact that $\gamma_i <
\beta_i$ implies a double maximal representation of $(\gamma_i
+1)g_i$; moreover, since $\gamma_i+1$ is the least integer with
this property, the other maximal representation does not involve
$g_i$. Explicitly, we have the relation
\begin{equation}\label{sbagli}
 (\gamma_i+1) g_i = \sum_{j \ne i} \eta_j g_j, \qquad \gamma_i+1= \sum_{j\ne i} \eta_j.
 \end{equation}
Let us substitute the relation just found in $(\mu_2,\ldots,\mu_\nu)$,
i.e. consider the tuple $(\mu'_2,\ldots,\mu'_\nu)$ where
\[ \mu'_i = \mu_i - \gamma_i -1, \qquad \mu'_j= \mu_j + \eta_j \quad \text{for } j \ne i\]
in particular by (\ref{sbagli}) $(\mu'_2,\ldots,\mu'_\nu) \in \mathcal{R}$.
Now, if there is a index $j<i$ such that $\eta_j \ne 0$,
then  $(\mu_2,\ldots,\mu_\nu) < (\mu'_2,\ldots,\mu'_\nu)$ with respect to \textsf{lex},
yielding a contradiction to the choice of $(\mu_2,\ldots,\mu_\nu)$.
Hence $\eta_j=0$ for each $j<i$.
But this is a contradiction to (\ref{sbagli}), since $g_i <g_j$ for $i<j$ and $\gamma_i+1= \sum_{j>i}\eta_j$.
The lemma is proved.
\end{proof}

\begin{cor}
Let $S$ be a numerical semigroup, then  $\ap(S) \subseteq \Gamma$.
\end{cor}

\begin{rem}\label{father}
It is straightforward to check that, under the notation of Lemma
\ref{dolls},  $(\mu_2, \ldots, \mu_\nu)$ is also the maximum with
respect to \textsf{grlex} in the set of (not necessarily maximal)
representations of $\omega$.
\end{rem}

We are interested in numerical semigroups for which the inclusions shown so far turn out to be equalities.

\begin{defin}\label{1} Let $S$ be a numerical semigroup:
\begin{enumerate}
  \item the semigroup $S$ has {\em $\beta$-rectangular Ap\'ery set} if
  $\ap(S)=B$;
  \item the semigroup $S$ has {\em $\gamma$-rectangular Ap\'ery set} if
  $\ap(S)=\Gamma$.
\end{enumerate}
\end{defin}

\begin{cor}\label{work}
We have the following implication:
\begin{center}
$\ap(S)$ is 
$\beta$-rectangular $\Rightarrow$ $\ap(S)$ is
$\gamma$-rectangular.
\end{center}
\end{cor}
\begin{proof}
It follows immediately by the inclusions $\ap(S) \subseteq \Gamma
\subseteq B$
\end{proof}

\begin{rem} It would be natural to introduce also another set of integers:
$\alpha_i =\max\{ h \in \mathbb{N} \, | \, hg_i \in \ap(S)\}$; it
is clear that $\alpha_i \geq \beta_i$. These integers, as for the
$\beta_i$'s and the $\gamma_i$'s, yield to another class of
semigroups (that we could call semigroups with
$\alpha$-rectangular Ap\'ery set) interesting to discuss and
somehow similar to the two classes just defined. However, this
class would provide a too strong  condition with respect to the
property of $\overline G$ to be Complete Intersection, hence, for
brevity, we will omit its study.
\end{rem}

\begin{exs}\label{face}

(1) Let $S= \langle 8, 10, 15 \rangle$; we have seen in Examples
\ref{liberi} that $\ap(S) = \{0, 10, 15,$ $ 20, 25, 30, 35, 45\}$
and $\beta_2=3,\, \beta_3=1$. Hence we obtain
that 
$\ap(S)$ is
$\beta$-rectangular.

(2) Let $S =\langle 8, 10, 11, 12\rangle$; we have $\ap(S) = \{0,
10, 11, 12, 21, 22, 23, 33\}$. Since $20, \, 44, \, 24 \notin
\ap(S)$ and $\ord(10)=1, \, \ord(33)=3, \, \ord(12)=1$ we have
$\beta_2= 1, \, \beta_4=3, \beta_4=1$. It is clear that
$\gamma_2=\gamma_4 = 1$; since $22= 11+11=10+12$ has two maximal
representation, we find $\gamma_3=1$. Thus $\ap(S)$ is
$\gamma$-rectangular  but $\ap(S)$ is not $\beta$-rectangular.

(3) Let $S=\langle 5, 6 , 9 \rangle$; we have
$\ap(S)=\{0,6,9,12,18\}$ and the only double representation is
$18=3\cdot 6=2 \cdot 9$. In this case $\gamma_2=3$, as $18=3\cdot
6$ is the unique maximal representation, while $\gamma_3=1$
because $\ord(2\cdot 18)=3 >2$. It follows that $\ap(S)$ is not
$\gamma$-rectangular, since $\lambda_2\cdot6+\lambda_3\cdot9
\notin \ap(S)$ as soon as both $\lambda_2 > 0$ and $\lambda_3 >
0$.
\end{exs}

Our main purpose is to characterize these kinds of numerical
semigroups in terms of some of their invariants, namely the
Frobenius number and the multiplicity. This will allow simpler
tests for those properties. We also want to study the relations
between these classes of semigroups and the properties of unique
maximal expression of $\ap(S)$ and of the partial orders $\preceq$
and $\preceq_M$.

The next lemmas concern maximal representations and are necessary
to establish a characterization of semigroups with
$\beta$-rectangular $\ap(S)$.

\begin{lem}\label{b}
If $s \preceq_M t$ and $t$ has a unique maximal representation,
then $s$ has a unique maximal representation.
\end{lem}

\begin{proof}
We have that $s+u= t$ for some $u \in S$ and
$\ord(s)+\ord(u)=\ord(t)$. If $s$ had two maximal representations,
the same should be true for $t$; contradiction.
\end{proof}

\begin{lem}\label{c}
Let $S$ be a numerical semigroup with $\beta$-rectangular Ap\'ery
set and let $\omega \in \ap(S)$. Then any representation
$\omega=\lambda _2g_2+\dots+\lambda_\nu g_\nu$, with $\lambda_i
\leq \beta_i$ for each $i=2,\dots,\nu$, is maximal.
\end{lem}

\begin{proof}
Assume that the representation $\omega=\lambda
_2g_2+\dots+\lambda_\nu g_\nu$ is not maximal. Let
$\omega=\mu_2g_2+\dots+\mu_\nu g_\nu$ be a maximal representation
of $\omega$; since $\omega\in \ap(S)$, by Lemma \ref{aurora}, we
have $\mu_i \leq \beta_i$ for every $i=2,\dots,\nu$. Moreover, by
maximality, we have $\sum_{i=2}^{\nu} \lambda_i < \sum_{i=2}^{\nu}
\mu_i$, hence  there exists an index $i$ such that
$\lambda_i<\mu_i$. On the other hand, since $\lambda
_2g_2+\dots+\lambda_\nu g_\nu=\mu_2g_2+\dots+\mu_\nu g_\nu$, there
exists another index $j$ such that $\lambda_j>\mu_j$. Subtracting
from both sides of the previous equality the common summands, we
get the equality
\begin{equation}\label{cresce}
\sum_{i\in T_1}\eta_i g_i=\sum_{i\in T_2}\eta_i g_i
\end{equation}
 where $T_1,\, T_2$ are two non-empty disjoint subsets of $\{2,\dots \nu\}$,  $0 \leq
\eta_i \leq \beta_i$, for each $i=2,\dots, \nu$, and $\sum_{i\in
T_1}\eta_i < \sum_{i\in T_2}\eta_i$.

Since $S$ has $\beta$-rectangular Ap\'ery set, the element
$t=\sum_{i=2}^{\nu}\beta_ig_i$ belongs to $\ap(S)$. Let us
substitute the relation (\ref{cresce}) in this representation of
$t$, that is to say consider the tuple $(\xi_2, \ldots, \xi_\nu)$
where
\[ \xi_i= \beta_i - \eta_i \mbox{ if } i \in T_1,
\qquad \xi_i= \beta_i + \eta_i \mbox{ if } i \in T_2, \qquad
\xi_i= \beta_i \mbox{ otherwise. } \] This is another
representation of $t$ and $\sum_{i=2}^\nu \beta_i < \sum_{i=2}^\nu
\xi_i$ by (\ref{cresce}). It follows that $(\beta_2, \ldots,
\beta_\nu)$ is not a maximal representation and, therefore, $t$
has a maximal representation with some coefficient bigger than
$\beta_i$; contradiction to Lemma \ref{aurora}.
\end{proof}

\begin{rem}\label{micro}
Combining the last lemma and Lemma \ref{aurora}, we have that the
maximality of a representation $(\lambda_2, \ldots, \lambda_\nu)$
of an element of $\ap(S)$ is equivalent to have  $\lambda_i \leq
\beta_i$ for each $i$, under the assumption of $\beta$-rectangular
Ap\'ery set.
\end{rem}

We are ready to give some characterizations of semigroups with
$\beta$-rectan\-gular Ap\'ery set.

\begin{thm}\label{d}
The following conditions are equivalent:

\begin{enumerate}
     \item[(i)]
      $\ap(S)$ is $\beta$-rectangular;
       \item[(ii)] $\ap(S)$ has a unique maximal element with respect to $\preceq_M$ and this element has
       unique maximal representation;
       \item[(iii)] $S$ is $M$-pure, symmetric and  $\ap(S)$ is of unique maximal expression;
       \item[(iv)] $f+m=\sum_{i=2}^{\nu} \beta_i g_i$;
       \item[(v)] $m=\prod_{i=2}^{\nu}(\beta_i+1)$.
\end{enumerate}
\end{thm}
\begin{proof}

(i) $\Rightarrow$ (ii) Since $\ap(S)$ is $\beta$-rectangular, we
immediately get that $\sum_{i=2}^{\nu} \beta_i g_i$ is the unique
maximal element in $(\ap(S), \preceq)$ and hence it is $f+m$.
Moreover, by Lemma \ref{c},
 any element $\omega \in \ap(S)$ is maximally represented by a tuple
 $(\lambda_2, \ldots, \lambda_\nu)$ with $\lambda_i \leq \beta_i$;
 in particular $\ord(\omega)= \sum_{i=2}^{\nu} \lambda_i$ and we can deduce that
 $\omega \preceq_M f+m$, for each $\omega\in \ap(S)$.

 Finally, by Remark \ref{micro}, $\sum_{i=2}^{\nu} \beta_i g_i$ is the unique maximal
representation of $f+m$.

(ii) $\Rightarrow$ (iii) By Proposition \ref{crede}, $S$ is
$M$-pure symmetric; by Lemma \ref{b}, every  $\omega \in \ap(S)$
has a unique maximal representation.

(iii) $\Rightarrow$ (iv) The unique maximal element in $(\ap(S),
\preceq_M)$ is necessarily $f+m$. Its unique maximal
representation is of the form $f+m=\sum_{i=2}^{\nu} \lambda_i g_i$
with $ \lambda_i \leq \beta_i$ by Lemma \ref{aurora}.

Since $S$ is $M$-pure symmetric $\beta_i g_i \preceq_M f+m$ for
each $i=2,\ldots, \nu$. Hence $f+m-\beta_2g_2 \in \ap(S)$ and, by
Lemma \ref{b}, it has unique maximal representation
$f+m-\beta_2g_2=\sum_{i=2}^{\nu} \varepsilon_i g_i$; moreover
$\varepsilon_i\leq \beta_i$,  by maximality of the representation.
It follows that $f+m=\beta_2g_2+\sum_{i=2}^{\nu} \varepsilon_i
g_i$ is a maximal representation, hence $\varepsilon_2=0$ (again
by Lemma \ref{aurora}). Moreover, since $f+m$ has a unique maximal
representation, $\lambda_2=\beta_2$; arguing recursively, the
thesis follows.

(iv) $\Rightarrow$ (i) It is a consequence of Proposition
\ref{1.1}.

(i) $\Rightarrow$ (v) It follows by $m=|\ap(S)|$, the fact that
$\ap(S)$ is of unique maximal expression
 and Remark \ref{micro}.

(v) $\Rightarrow$ (i) We already noticed that $\ap(S) \subseteq B$
and since $m=\prod_{i=2}^{\nu}(\beta_i+1)=|\ap(S)|$, we must have
an equality.
\end{proof}

\begin{ex}\label{beta}
We may apply the last theorem to show that the Ap\'ery set of
$S=\langle 12,14,16,23 \rangle$ is $\beta$-rectangular, without
even computing the whole set $\ap(S)$. We need to determine the
$\beta_i$'s:
\begin{eqnarray*}
2\cdot 14 &=& 12+ 16 \in 12+S \\
2 \cdot 16 &=& 32 \in 2M \setminus 3M \ \text{and} \ 32-12 \notin S \\
3 \cdot 16 &=& 4 \cdot 12 \in 12+S\\
2 \cdot 23 &=& 2 \cdot 16+14 \in 3M
\end{eqnarray*}
and so $\beta_2=1,\, \beta_3=2,\, \beta_4=1 $ and $m=12=
2\cdot3\cdot2 =\prod_{i=2}^{\nu}(\beta_i+1)$.

\end{ex}
\medskip

We have seen (cf. Corollary \ref{work}) that the following
implication holds:
\begin{center}
$\ap(S)$ is
$\beta$-rectangular $\Rightarrow$ $\ap(S)$ is
$\gamma$-rectangular.
\end{center}
But a priori we still might have
$\beta_i > \gamma_i$ when
$\ap(S)$ is $\beta$-rectangular, for some $i$. We show that this
is not the case, using the theorem just proved.
\begin{cor}\label{beta2}
Let $S$ be a numerical semigroup. If $\ap(S)$ is
$\beta$-rectangular, then $\beta_i = \gamma_i$, for every $i = 2 ,
\ldots, \nu$.
\end{cor}
\begin{proof}
If there is an index $i$ such that $\gamma_i < \beta_i $,
then, by definition of $\beta_i$ and of  $\gamma_i$, $\beta_i g_i$
is in the Ap\'ery set and it has more maximal representations.
Contradiction to Theorem \ref{d}, (ii).
\end{proof}
\medskip

We now turn to the study of semigroups with $\gamma$-rectangular
Ap\'ery set, starting with a result that is analogous to Lemma
\ref{c}.

\begin{lem}\label{soul}
Let $S$ be a semigroup with $\gamma$-rectangular Ap\'ery set and
let $\omega \in \ap(S)$. Then any representation $\omega=
\lambda_2 g_2 + \cdots + \lambda_\nu g_\nu$, with $\lambda_i \leq
\gamma_i$ (for every $i=2\dots,\nu$), is maximal.
\end{lem}
\begin{proof}
Assume, by absurd, that there exists  $\omega\in \ap(S)$ with a
non-maximal representation $\omega= \sum_{i=2}^\nu \lambda_i g_i$,
where $\lambda_i \leq \gamma_i$ (for every $i=2\dots,\nu$). Notice
that, as we have already seen for the case $\ap(S)=B$, if
$\ap(S)=\Gamma$, then $f+m= \sum_{i=2}^\nu \gamma_i g_i$. In
particular, up to substituting the non-maximal representation of
$\omega$ in $f+m=\sum_{i=2}^\nu \gamma_i g_i$, we may assume
$\omega=f+m$ and hence $\lambda_i = \gamma_i$, for each $i$.

If we take a maximal representation $(\mu_2, \ldots, \mu_\nu)$ of
$f+m$, then we have the strict inequality
\begin{equation}\label{asiam}
(\gamma_2, \ldots, \gamma_\nu) < (\mu_2, \ldots, \mu_\nu)
\end{equation}
with respect to \textsf{grlex} (since the sums of the respective
coefficients are different). In particular (\ref{asiam}) holds if
we choose $(\mu_2, \ldots, \mu_\nu)$ to be the maximum in the set
of all the representations of $f+m$ with respect to
\textsf{grlex}. By Remark \ref{father}  we have
\[(\gamma_2, \ldots, \gamma_\nu) < (\mu_2, \ldots, \mu_\nu) \leq (\gamma_2, \ldots, \gamma_\nu)\]
with respect to \textsf{grlex}, and thus we reach a contradiction.
\end{proof}

\begin{lem}\label{honor}
Let $S$ be a semigroup with $\gamma$-rectangular Ap\'ery set. Then
each $\omega \in \ap(S)$ has a unique representation of the form
$\omega= \lambda_2 g_2 + \cdots + \lambda_\nu g_\nu$, with
$\lambda_i\leq \gamma_i$, for every $i=2,\dots \nu$.
\end{lem}
\begin{proof}
Assume, by absurd,
that there are two distinct representations of $\omega \in \ap(S)$
\[ \omega= \sum_{i=2}^\nu \lambda_i g_i = \sum_{i=2}^\nu \mu_i g_i, \text{ with } \lambda_i, \mu_i \leq \gamma_i. \]
By Lemma \ref{soul} both representations are maximal, and in
particular $\sum_{i=2}^\nu \lambda_i  = \sum_{i=2}^\nu \mu_i$;
since they are distinct we have, for instance,
\[ (\lambda_2, \ldots, \lambda_\nu) < (\mu_2, \ldots, \mu_\nu)\]
with respect to \textsf{lex}. By adding the tuple $(\gamma_2
-\lambda_2, \ldots, \gamma_\nu - \lambda_\nu)$ to both sides of
the last equality we have two maximal representations of $f+m$:
\[ (\gamma_2, \ldots, \gamma_\nu) <  (\mu_2+\gamma_2 -\lambda_2, \ldots,\mu_\nu+ \gamma_\nu - \lambda_\nu)\]
with respect to \textsf{lex}.
We reach an absurd by Lemma \ref{dolls}.
\end{proof}

\begin{cor}\label{joy}
If $S$ has  $\gamma$-rectangular Ap\'ery set, then it is $M$-pure symmetric.
\end{cor}
\begin{proof}
Let $\omega \in \ap(S)$. By Lemma \ref{dolls}, $\omega=
\sum_{i=2}^\nu \lambda_i g_i$, with $\lambda_i \leq \gamma_i$,
and, by Lemma \ref{soul}, this representation is maximal;
therefore $\ord(\omega)=\sum_{i=2}^\nu \lambda_i$. Since this is
valid for an arbitrary $\omega \in \ap(S)$, we easily have $\omega
\preceq_M \sum_{i=2}^\nu \gamma_i g_i=f+m$ and the thesis follows
by Proposition \ref{crede}.
\end{proof}

We are ready to characterize semigroups with $\gamma$-rectangular
Ap\'ery set.

\begin{thm}\label{charact}
The following conditions are equivalent:
\begin{enumerate}
     \item[(i)]
      $\ap(S)$ is $\gamma$-rectangular;
     \item[(ii)] $f+m=\sum_{i=2}^{\nu} \gamma_i g_i$;
     \item[(iii)] $m=\prod_{i=2}^{\nu}(\gamma_i+1)$.
\end{enumerate}
\end{thm}
\begin{proof}
(i) $\Rightarrow$ (ii)
Clear, as $f+m$ is the biggest element in $\ap(S)$.

(ii) $\Rightarrow$ (i)
It follows by Lemma \ref{1.1}.

(i) $\Rightarrow$ (iii) It follows by Lemma \ref{honor}.

(iii) $\Rightarrow$ (i) Since $\ap(S) \subseteq \Gamma$ and
$|\ap(S)|=m= \prod_{i=2}^{\nu}(\gamma_i+1)\geq |\Gamma|$, the
inclusion must be an equality.
\end{proof}

Notice that we do not obtain a full analogous result  of 
Theorem \ref{d}; more precisely, we cannot recover conditions (ii)
and (iii). The closest analogous to those conditions is actually
expressed by Lemma \ref{honor}. If we look at the second semigroup
in Examples \ref{face}, $S =\langle 8, 10, 11, 12\rangle$ (with
$\ap(S) = \{0, 10, 11, 12, 21, 22, 23, 33\}$ and $\gamma_i=1$, for
each $i=2,3,4$), we notice that both $22$ and $33$ have two
maximal representations, but only one in $\Gamma$.

\section{The Main Theorem}

In this section, we apply the results contained in the previous
section in order to give a characterization of the numerical semigroup
rings whose associated graded ring is Complete Intersection.\\

We recall that in the first section we defined
$\overline{R}=R/(t^m)$ and $\overline{G}= \gr_{\overline{\mathfrak
m}} (\overline{R})$, where ${\overline{\mathfrak m}}$ is the
maximal ideal of $\overline{R}$.

We also need to introduce two more invariants associated to $R$.
The ideal $Q=(t^m)$ is a principal reduction of the maximal
ideal $\mathfrak{m}$, that is a principal ideal $Q\subseteq
\mathfrak{m}$ such that $Q \mathfrak{m}^h= \mathfrak{m}^{h+1}$ for
some non-negative integer $h$. The \emph{reduction number} is the
integer $r=r_Q(\mathfrak{m})= \min \{ h \in \mathbb{N}, \,
Q\mathfrak{m}^h=\mathfrak{m}^{h+1}\}$, while the \emph{index of
nilpotency} is defined as $s=s_Q(\mathfrak{m})=\min\{ h \in
\mathbb{N}, \, \mathfrak{m}^{h+1} \subseteq Q\}$. In the case of
numerical semigroup rings we have $r= \min \{ h \in \mathbb{N}, \,
m+hM=(h+1)M\}$ and $s=\max\{\ord(\omega_i) \, | \, \omega_i \in \ap(S)\}$;
from the last two equalities it is easy to see that $s \leq r$.\\

We will also need the following result of Bryant.

\begin{thm}[\cite{Br},Theorem 3.14]\label{Bry}
Under the above notation, we have:

(1) $S$ is $M$-pure symmetric if and only if $\overline{G}$ is
Gorenstein;

(2) if $ \gr_{\mathfrak{m}}(R)$ is Cohen-Macaulay, then $s=r$.
The converse holds if $S$ is $M$-pure;

(3) $ \gr_{\mathfrak{m}}(R)$ is Gorenstein if and only if $S$ is
$M$-pure symmetric and $s=r$.

\end{thm}

To prove the main result of the paper we have to be more precise
about the generators of the ideal $J$ defined in Discussion
\ref{monomials}.

\begin{disc}\label{monomials2}
Using the terminology introduced in the previous section, in
Discussion \ref{monomials} we have shown that $J$ is generated by
all the the binomials of the form
\begin{equation}\tag{$+$}
x_2^{j_2}\cdot \dots \cdot x_{\nu}^{j_{\nu}}-x_2^{h_2}\cdot \dots
\cdot x_{\nu}^{h_{\nu}},
\end{equation}
 where $j_2g_2+\dots+j_{\nu}g_{\nu}=
h_2g_2+\dots+h_{\nu}g_{\nu} \in \ap(S)$ are two maximal
representations, and by all the monomials of the form
\begin{equation}\tag{$++$}
x_2^{j_2}\cdot \dots \cdot x_{\nu}^{j_{\nu}},
\end{equation}
where either
$j_2g_2+\dots+j_{\nu}g_{\nu} \notin \ap(S)$ or
$j_2g_2+\dots+j_{\nu}g_{\nu} \in \ap(S)$ and it is not a maximal
representation.

By definition of $\beta_i$ it follows that $x_i^{\beta_i+1}\in J$.
In fact, if $\ord(\beta_i+1)g_i > \beta_i+1$, then
$(t^{g_i})^{\beta_i+1}\in \mathfrak{m}^{\beta_i+2}$ and therefore
$(\overline{t^{g_i}})^{\beta_i+1}=0$ in $\gr_\mathfrak{m}(R)$ and
hence in $\overline G$. On the other hand, if $(\beta_i+1)g_i
\notin \ap(S)$, then $(\overline{t^{g_i}})^{\beta_i+1}=0$ in
$\overline R$ and hence in $\overline G$. Moreover, by definition
of $\beta_i$ and by Discussion \ref{monomials} it is clear that
$x_i^{\beta_i}\notin J$ for every index $i$.

On the other hand, by definition of $\gamma_i$ we have that,
$\gamma_i < \beta_i$ if and only if $(\gamma_i+1)g_i \in \ap(S)$
is a maximal representation, but it is not unique. Hence
$(\gamma_i+1)g_i=\sum_{j\neq i}\lambda_jg_j$ and
$\gamma_i+1=\sum_{j\neq i}\lambda_j$; equivalently
$x_i^{\gamma_i+1}-\prod_{j\neq i}x_j^{\lambda_j} \in J$.

Notice that, for some $h\leq \beta_i$ (hence, also for some $h \leq
\gamma_i$), it could happen that $hg_i=\sum_{j\neq i}\lambda_jg_j$
and $h>\sum_{j\neq i}\lambda_j$; in this case, $\prod_{j\neq
i}x_j^{\lambda_j} \in J$. Hence the smallest pure power of $x_i$
appearing in a monomial or a binomial of $J$ is
$x_i^{\gamma_i+1}$.


Finally, it is clear that, if a binomial $(+)$ is in $J$, then we
can cancel all the common factors: in fact, by Proposition
\ref{1.1}, if $\omega \in \ap(S)$ and $u \in S$ is such that $u
\preceq \omega$, then $u \in \ap(S)$.
\end{disc}

In the next corollary we summarize the results of Discussions
\ref{monomials} and \ref{monomials2}, that we will need in the
rest of the paper.

\begin{cor}\label{monomials3}  For every  $i=2,\dots,\nu$, we
have:
\begin{enumerate}
     \item[(i)]
        $x_i^{\beta_i+1}\in J$ and $x_i^{\beta_i}\notin J$;
     \item[(ii)] $\gamma_i < \beta_i \
     \Longleftrightarrow \  (\gamma_i+1)g_i=\sum_{j\neq i}\lambda_jg_j$ and
$\gamma_i+1=\sum_{j\neq i}\lambda_j \ \Longleftrightarrow \
x_i^{\gamma_i+1}-\prod_{j\neq i}x_j^{\lambda_j} \in J$;
     \item[(iii)] the smallest pure power of $x_i$
appearing in a monomial or a binomial of $J$ is
$x_i^{\gamma_i+1}$.
\item[(iv)] $x_2^{\lambda_2}\cdot \dots \cdot
x_{\nu}^{\lambda_{\nu}}\notin J \ \Longleftrightarrow \
\sum_{j=2}^{\nu}\lambda_jg_j \in \ap(S)$ and
$\sum_{j=2}^{\nu}\lambda_jg_j$ is a maximal representation.
\end{enumerate}
\end{cor}

The next result is a key step in order to prove the main theorem.

\begin{lem}\label{JCI}
The ring $\overline G$ is Complete Intersection if and only if the
defining ideal $J$ is of the following form
$$J=(x_i^{\gamma_i+1}-\rho_i\prod_{j\neq
i}x_j^{\lambda_j} \ :\ i=2\dots,\nu ),$$ where $\rho_i=0$, if
$\beta_i=\gamma_i$ and $1$ otherwise; as soon as $\rho_i=1$,
$(\gamma_i+1)g_i=\sum_{j\neq i}\lambda_jg_j$ and
$\gamma_i+1=\sum_{j\neq i}\lambda_j$.
\end{lem}

\begin{proof}
We have that $\overline{G} \cong \Bbb\Bbbk[x_2, x_3, \dots, x_\nu] / J$;
hence, if $J$ is of the form described in the statement,
$\mu(J)=\nu-1$ and $\overline G$ is Complete Intersection.

Conversely, assume that $\overline{G}$ is Complete Intersection,
that is $\mu(J)= \nu -1$. By Corollary \ref{monomials3} (i), we
know that $(x_2^{\beta_2+1}, \dots, x_\nu^{\beta_\nu +1 } )
\subseteq J$ and that $x_i^{\beta_i}\notin J$ for every index $i$.

Hence, if $J\supsetneq (x_2^{\beta_2+1}, \dots, x_\nu^{\beta_\nu
+1 } )$, for every index $i$ such that $x_i^{\beta_i+1}$ is not a
minimal generator, there exists a unique binomial of the form
$x_i^h-\prod_{j\neq i}x_j^{h_j}$, which is a minimal generator. By
Corollary \ref{monomials3} (ii) and (iii), we have that $\gamma_i
<\beta_i$ and the generator is $x_i^{\gamma_i+1}-\prod_{j\neq
i}x_i^{\lambda_j}$ with $\gamma_i+1=\sum_{j\neq i}\lambda_j$,
since there is not any binomial in $J$ involving a pure power of
$x_i$ with exponent smaller than $\gamma_i+1$.
\end{proof}

\begin{rems}\label{remarks}
(1) Let $\tilde J=(x_i^{\gamma_i+1}-\rho_i\prod_{j\neq
i}x_j^{\lambda_j}:\ i=2\dots,\nu )$ (with the same notation of the
previous lemma); by Corollary \ref{monomials3}, it is clear that
we always have the inclusion $J\supseteq \tilde J$.
\\
(2) We can assume that in the binomials appearing as generators of
$\tilde J$ every exponent $\lambda_j$ is less than or equal to
$\gamma_j$: choose the biggest $(\nu-2)$-tuple $(\lambda_2,\dots
\lambda_{i-1},\lambda_{i+1}, \dots \lambda_{\nu})$ with respect to
\textsf{lex}, among all the possible $(\nu-2)$-tuples
corresponding to the maximal representations of $(\gamma_i+1)g_i$
and then argue in a similar way as in Lemma \ref{dolls}.

\noindent (3) Since $g_2 < g_3< \dots < g_{\nu}$, necessarily
$\rho_2=\rho_{\nu}=0$, that is $x_2^{\gamma_2+1}$ and
$x_{\nu}^{\gamma_{\nu}+1}$ are minimal generators of $\tilde J$.
\end{rems}

\begin{thm}\label{7}
The following conditions are equivalent:

\begin{enumerate}
     \item[(i)]
     $\gr_\mathfrak{m}(R)$ is Complete Intersection;
     \item[(ii)]
      $\ap(S)$ is $\gamma$-rectangular and $\gr_\mathfrak{m}(R)$ is Cohen-Macaulay;
     \item[(iii)]
      $\ap(S)$ is $\gamma$-rectangular and $r=s$;
     \item[(iv)]
      $\ap(S)$ is $\gamma$-rectangular and $\gr_\mathfrak{m}(R)$ is Buchsbaum;
     \item[(v)]
      $\ap(S)$ is $\gamma$-rectangular and $\gr_\mathfrak{m}(R)$ is
      Gorenstein.
     \end{enumerate}
\end{thm}

\begin{proof} (i) $\Rightarrow$ (ii) Clearly $\gr_\mathfrak{m}(R)$ is Cohen-Macaulay.
By Remark \ref{4} we know that $\overline{G}$ is Complete
Intersection. Using the previous lemma, we know that $J$ is
generated by $\nu-1$ forms of degree $\gamma_i+1$, for every
$i=2,\dots,\nu$.

Since $\overline G\cong \Bbb\Bbbk[x_2,\dots,x_{\nu}]/J$, we have
dim$_\Bbbk(\overline G)=\prod_{i=2}^{\nu}(\gamma_i+1)$, hence every
monomial $x_2^{\lambda_2}\dots x_{\nu}^{\lambda_{\nu}}$, with
$\lambda_i \leq \gamma_i$, does not belong to $J$ and their images
are pairwise different in $\overline G$. By Corollary
\ref{monomials3} (iv), the corresponding elements
$\lambda_2g_2+\dots+\lambda_{\nu}g_{\nu}$ of $S$ belong to
$\ap(S)$ (and all of these are maximal representations).
Hence we obtain $ \Gamma= \{ \lambda_2 g_2 + \lambda_3 g_3 + \dots
+ \lambda_\nu g_{\nu} \, | \, \lambda_i =0, \ldots, \gamma_i, \,
i=2,\dots, \nu\}=\ap(S)$, that is $\ap(S)$ is
$\gamma$-rectangular.

(ii) $\Rightarrow$ (i) $\gr_\mathfrak{m}(R)$ is Cohen-Macaulay
and, by \cite[Theorem 7]{Ga}, $({t^m})^*$ is regular. Hence, by
Remark \ref{4}, $G$ is Complete Intersection if and only if
$\overline{G}$ is Complete Intersection.

We know that, as $\Bbbk$-vector space, $\overline{G}= \langle
\overline{t^{\omega_i}} \, | \, \omega_i \in \ap(S)\rangle_\Bbbk$.
Moreover, by Remarks \ref{remarks} (1), it is clear that $J
\supseteq \tilde J$, hence
$$m=|\ap(S)|=\dim_\Bbbk(\overline{G}) \leq
\dim_\Bbbk(\Bbbk[x_2,\dots,x_{\nu}]/\tilde J)=
\prod_{i=2}^{\nu}(\gamma_i+1) \ .$$ Since $\ap(S)$ is
$\gamma$-rectangular, by Theorem \ref{charact},
$m=\prod_{i=2}^{\nu}(\gamma_i+1)$; thus in the above chain we have
all equalities and, therefore, $J=\tilde J$, that is $\overline G$
is Complete Intersection.

(ii) $\Leftrightarrow$ (iii) By Corollary \ref{joy}, $S$ is
$M$-pure; under this hypothesis $G$ is Cohen-Macaulay if and only
if $r=s$ (Theorem \ref{Bry}).

(ii) $\Leftrightarrow$ (iv) By \cite[Proposition 5.5]{DMS}.

(iv) $\Leftrightarrow$ (v) By \cite[Corollary 5.6]{DMS}.
\end{proof}

\begin{rem}
Using the last theorem, in order to know if $\gr_\mathfrak{m}(R)$
is Complete Intersection we have to check if $\ap(S)$ is
$\gamma$-rectangular and if $\gr_\mathfrak{m}(R)$ is
Cohen-Macaulay (or Buchsbaum, or Gorenstein). In \cite[Theorem
2.6]{BF} there is a characterization of the Cohen-Macaulayness of
$\gr_\mathfrak{m}(R)$ that has been strengthened in
\cite[Proposition 5.1]{DMS}; in particular, in case $\ap(S)$ is
$\gamma$-rectangular, one has to compute the integers $a_i$ and
$b_i$ defined in these characterizations, only for $\omega_i=f+m$,
which is the only element in $\mapM(S)$. Notice also that, in this
case, to check the Buchsbaumness of $\gr_\mathfrak{m}(R)$ it is
not easier than to check the Cohen-Macaulayness (cf.
\cite[Proposition 3.6]{DMS}).
\end{rem}

\begin{exs}
Let us consider the semigroups (2) and (3) of the Examples
\ref{face}. \\
 (2) $S=\langle 8,10,11,12 \rangle$: here $\ap(S)$ is
$\gamma$-rectangular (and not $\beta$-rectangular) and
$\gr_\mathfrak{m}(R)$ is Cohen-Macaulay, since $r=3=\ord(33)$.
Hence $\gr_\mathfrak{m}(R)$ is Complete Intersection. Computing
the defining ideals we get:
$I=(x_2^2-x_1x_4,x_3^2-x_2x_4,x_1^3-x_4^2)$,
$I^*=(x_2^2-x_1x_4,x_3^2-x_2x_4,x_4^2)$ and
$J=(x_2^2,x_3^2-x_2x_4,x_4^2)$.
\\
(3) $S=\langle 5,6,9\rangle$: here $\ap(S)$ is not
$\gamma$-rectangular and $\gr_\mathfrak{m}(R)$ is Cohen-Macaulay
(as can be checked using \cite[Proposition 5.1]{DMS}); $S$ is
symmetric, but not $M$-pure (since $9 \not\preceq_M 18$).
Therefore, $\gr_\mathfrak{m}(R)$ is not Complete Intersection (nor
Gorenstein). Computing he defining ideals we obtain:
$I=(x_1^3-x_2x_3,x_2^3-x_3^2)$, $I^*=(x_2x_3,x_3^2,
x_2^4-x_1^3x_3)$ and $J=(x_3^2,x_2x_3,x_2^4)$.
\end{exs}

\begin{rem}
The proof of equivalence (i) $\Leftrightarrow$ (ii) of Theorem
\ref{7} shows that $\overline{G} = \gr_{\overline{\mathfrak{m}}}
(\overline{R})$ is Complete Intersection if and only if $\ap(S)$
is $\gamma$-rectangular. If we want to extend the result to
$\gr_\mathfrak{m}(R)$
the hypothesis of the Cohen-Macaulayness
is indispensable, that is
$$
\ap(S)\ \mbox{is} \ \gamma-\mbox{rectangular } \
  \not\Rightarrow \ \gr_\mathfrak{m}(R)\
\mbox{is Complete Intersection} .
$$
Indeed, let $S=\langle 6,7,15 \rangle $. Here $\ap(S)$ is
$\beta$-rectangular (hence $\gamma$-rectangular), since $f=23, \,
\beta_2=2, \beta_3=1$ and $f+m=\beta_2 g_2 + \beta_3 g_3$. It
follows that $\overline G$ is Complete Intersection and
$J=(x_3^2,x_2^3)$. However, in this case, $G$ is not Complete
Intersection, since $\gr_\mathfrak{m}(R)$ is not Cohen-Macaulay,
because $s=3$ and  $r=6$. It is not difficult to compute that
$I=(x_3^2-x_1^5, x_2^3-x_1x_3)$ ($S$ is symmetric, hence $R$ is
Gorenstein that, if the embedding dimension is $3$, implies Complete
Intersection) and that $I^*=(x_3^2, x_1x_3, x_2^6)$.
\end{rem}

\begin{rem}
We note that:
\begin{center}
$\gr_\mathfrak{m}(R)$ is Gorenstein and $R$ is Complete
Intersection $\not\Rightarrow$ \\ $\gr_\mathfrak{m}(R)$ is
Complete Intersection.
\end{center}
In fact, let $S=\langle 16,18,21,27 \rangle$; we have $\ap(S)= \{
0, 18, 21, 27, 36, 39, 42, 45, \linebreak 54, 57, 60, 63, 72, 78,
81, 99 \}$. Let us compute the integer $\gamma_3$: $2\cdot 21=42
\in \ap(S)$ and it has a unique representation, while $3\cdot 21=
2 \cdot 18+27$ are two maximal representations of $63$; hence
$\gamma_3=2$. The multiplicity of $S$ is $m=16$ and $\gamma_3+1=3$
does not divide $m$; therefore, by Theorem \ref{charact}, $\ap(S)$
is not $\gamma$-rectangular (hence $\gr_\mathfrak{m}(R)$ is not
Complete Intersection).

On the other hand, it is not difficult to check that $S$ is
symmetric, $M$-pure and that $\ord(99)=5=r$; hence, by Theorem
\ref{Bry} (3), $\gr_\mathfrak{m}(R)$ is Gorenstein. Finally, $R$
is Complete Intersection, since $I=(x_1^3-x_3x_4, x_2^3-x_4^2,
x_2^2x_4-x_3^3)$.
\end{rem}

\begin{rem}\label{Jmon}
If we substitute the condition ``$\gamma$-rectangular Ap\'ery set"
with the condition ``$\beta$-rectangular Ap\'ery set" in the
previous Theorem \ref{7} we obtain not only that
$\gr_\mathfrak{m}(R)$ is Complete Intersection, but also that $J$
is monomial. However this fact does not implies that the defining
ideal $I^*$ of $\gr_\mathfrak{m}(R)$ is monomial. For example, let
us consider the semigroup of the Example \ref{beta}: $S=\langle
12,14,16,23 \rangle$; its Ap\'ery set is $\beta$-rectangular and
$r=s=4$, hence $\gr_\mathfrak{m}(R)$ is Complete Intersection. In this case
the defining ideals are the following:
$I=(x_2^2-x_1x_3,x_4^2-x_2x_3^2,x_3^3-x_1^4)$,
$I^*=(x_2^2-x_1x_3,x_4^2,x_3^3)$ and $J=(x_2^2,x_3^3,x_4^2)$.
%
%
%
\end{rem}

As a corollary to Theorem \ref{7}, we get the following corollary,
that can be also obtained easily by B\'ezout's theorem applied to
the algebraic variety defined by $G$. Given a positive integer
$x$, we call $\ell(x)$ the length of its unique factorization i.e.
the number of (possibly equal) prime factors of $x$.

\begin{cor}
Let $\gr_\mathfrak{m}(R)$ be Complete Intersection, then we have:

\begin{enumerate}
     \item[(1)]
     $\nu \leq \ell(m) +1$;
     \item[(2)]
     if $m$ is a prime number then $\nu=2$.
\end{enumerate}
\end{cor}
\begin{proof}

(1) By Theorem \ref{7}, $\ap(S)$ is $\gamma$-rectangular. Thus we have
$m=|\ap(S)|=\prod_{i=2}^{\nu} (\gamma_i+1)$ and the thesis follows as
$\gamma_i\geq 1$.
\\
(2) If $m$ is prime, we get $\nu \leq \ell(m)+1= 2$ by (1);
but $\nu > 1$, otherwise $m=1$.
\end{proof}

Using our method, we also obtain an alternative proof of a result
from \cite{BF}:

\begin{cor}\label{7a}
Let $n_{\nu} < \cdots < n_1$ be pairwise relatively prime positive integers, $N= \prod_{i=1}^{\nu} n_i$.
Let $g_i=\frac{N}{n_i} $ and $S=\langle g_1, g_2, \ldots, g_{\nu} \rangle$.
Then $\gr_\mathfrak{m}(R)$ is Complete Intersection.
\end{cor}

\begin{proof}
As proved in \cite[Proposition 3.6]{BF}, $\gr_\mathfrak{m}(R)$ is Cohen-Macaulay;
 we only need to prove that $\ap(S)$ is $\gamma$-rectangular.
We easily have $\gamma_i\leq n_i-1$ and hence $\prod_{i=2}^{\nu}
(\gamma_i+1)\leq \prod_{i=2}^{\nu} n_i=g_1=m$.  Since $m=|\ap(S)|
\leq |\Gamma|\leq \prod_{i=2}^{\nu} (\gamma_i+1)$, the equality
$m= \prod_{i=2}^{\nu} (\gamma_i+1)$ holds and, by Theorem
\ref{charact}, we get the thesis.
\end{proof}

We finish the paper studying the case $\nu(S)=3$. By Remarks
\ref{remarks} (3), it follows immediately that $\ap(S)$ is
$\gamma$-rectangular if and only if it is $\beta$-rectangular.
However, we can prove something more.

\begin{thm}\label{3gen}
Let $S= \langle g_1, g_2, g_3 \rangle$ be a three-generated
semigroup (with $g_1 < g_2 < g_3$). The following conditions are
equivalent:
\begin{enumerate}
\item[(i)] $\ap(S)$ is $\beta$-rectangular;

\item[(ii)] $\ap(S)$ is $\gamma$-rectangular;

\item[(iii)] $S$ is $M$-pure symmetric.
\end{enumerate}
\end{thm}
\begin{proof}
By Corollary \ref{work} and Corollary \ref{joy} we only need to
show the implication $(iii) \Rightarrow (i)$. Assume that $S$ is
$M$-pure symmetric, that is $\omega \preceq_M f+m$ for each
$\omega \in \ap(S)$. If we have two maximal representations of
$f+m$:
\[ f+m= \lambda_2 g_2 + \lambda_3 g_3 = \mu_2 g_2 + \mu_3 g_3,\qquad \mbox{ with } \lambda_2 + \lambda_3 = \mu_2 + \mu_3\]
then it follows that
\[ (\lambda_2 - \mu_2) g_2= (\mu_3-\lambda_3) g_3, \qquad \mbox{ with } \lambda_2 - \mu_2 = \mu_3-\lambda_3\]
and, since $g_2<g_3$, $\lambda_2-\mu_2=\lambda_3-\mu_3=0$. Thus
$f+m$ has a unique maximal representation and the thesis follows
by Theorem \ref{d} (iii).
\end{proof}

As a consequence of the last result, we get the well known fact
that Goresteinness and Complete Intersection are equivalent in codimension
two
 (i.e., in our hypotheses, in embedding dimension three).

\begin{cor}
Let $R$ be a numerical semigroup ring with $\nu(R)=3$. If
$\gr_\mathfrak{m}(R)$ is Gorenstein, then it is Complete
Intersection (and the ideal $J$ is monomial).
\end{cor}
\begin{proof}
Apply Theorems \ref{Bry}, \ref{3gen} and \ref{7}.
\end{proof}

\textbf{Acknowledgments.} The authors are grateful to Yi Huang
Shen for some helpful remarks on a preliminary version of the
paper that allowed to define properly the $\gamma$-rectangularity
and to the referee for several suggestions and comments that
improved the paper.

\end{document}